\documentclass[11pt]{article}

\usepackage{amsmath,amssymb,amsthm,amscd,epsfig}
\newcommand{\vp}{\varepsilon}

\theoremstyle{plain}

\newtheorem{thm}{Theorem}

\newtheorem*{unthm}{Theorem}

\newtheorem{cor}{Corollary}

\newtheorem{pro}{Proposition}

\theoremstyle{definition}

\theoremstyle{remark}

\begin{document}

\title{Unstable directions and dimension for a class of skew products with overlaps}

\author{Eugen Mihailescu}

\date{}
\maketitle

\begin{abstract}

A unique and distinguishing feature of smooth hyperbolic non-invertible maps is that of having different unstable directions corresponding to different prehistories of the same point. In this
paper we construct a new class of examples of non-invertible
hyperbolic skew products with thick fibers for which we prove that
there exist uncountably many points in the locally maximal
invariant set $\Lambda$ (actually a Cantor set in each fiber),
having different unstable directions corresponding to different
prehistories; also we estimate the angle between such unstable
directions. We discuss then the Hausdorff dimension of the fibers
of $\Lambda$ for these maps by employing the thickness of Cantor sets, the inverse pressure, and also by use of continuous bounds for the preimage
counting function. We prove that in certain examples, there are uncountably
many points in $\Lambda$ with two preimages belonging to $\Lambda$, as well
as uncountably many points having only one preimage in $\Lambda$. In the end we give examples which, also from the point of view of Hausdorff dimension, are far from being homeomorphisms on $\Lambda$, as well as far from being constant-to-1 maps on $\Lambda$.

\end{abstract}



\section{Introduction}\label{sec1}

In this paper we study the chaotic dynamics of a class of non-invertible maps which are hyperbolic on their basic set of saddle type.
The dynamics of non-invertible maps (or \textit{endomorphisms}) is different than that of diffeomorphisms due to possible complicated overlappings
and to the fact that the number of preimages of a given point that remain in the respective basic set may vary.
Hyperbolicity over a compact invariant set $\Lambda$ can be defined for non-invertible smooth maps $f:M \to M$
(where $M$ is a manifold), as the continuous uniform splitting (\cite{R}) of the tangent space $T_{\hat x} M$
into a stable
subspace $E^s_x$ and an unstable subspace $E^u_{\hat x}$ for $\hat x \in \hat \Lambda$, where
$\hat x$ denotes a prehistory of $x$ (i.e a sequence of consecutive preimages of $x$ belonging to $\Lambda$)
and $\hat \Lambda$ is the space of all such sequences.
This implies that, unlike in the case of diffeomorphisms, where the local unstable manifolds form a foliation,
in the non-invertible case there may pass a priori several unstable manifolds
through a given point of $\Lambda$.
However it is a difficult problem to distinguish between unstable directions corresponding to different prehistories, and there are few actual examples of such hyperbolic endomorphisms for which several unstable manifolds pass through a given point.
One such example of an Anosov endomorphism was studied in \cite{Pr}, where also other properties of Anosov
endomorphisms were given; in fact for any arbitrary fixed point $x$ from the torus $\mathbb T^m$, it was proved that
there exists an endomorphism $f$, obtained as a perturbation of an algebraic hyperbolic endomorphism of $\mathbb T^m$,
 so that there exists infinitely many local unstable manifolds of $f$ through $x$.

In our paper we propose a way of obtaining endomorphisms with a
new type of strange behaviour, namely they are far from being
homeomorphisms (do not have an Open Set Condition-type behaviour),
and also far from being constant-to-one; this happens both from
the point of view of preimages, and from the point of view of
Hausdorff dimension. We study
 hyperbolic non-invertible skew products with basic sets whose fibers will be obtained by contractions followed by
 translations and superpositions, thus generating overlappings in fibers.
 Here by \textbf{basic set} for an endomorphism $f$ we mean an
 invariant compact set $\Lambda$ so that there exists a
 neighbourhood $U$ of $\Lambda$ with $\Lambda =
 \mathop{\cap}\limits_{n \in \mathbb Z} f^n(U)$.
We prove that, when the contraction factors are all equal to $\frac 12$ and the other parameters belong to some open set, we obtain a class of
examples in (\ref{skew-general}) which are \textbf{far} from having a homeomorphism-type
behaviour on their basic set $\Lambda$.
Our study will involve some new techniques, like the use of thickness of intersections of Cantor sets in fibers, the inverse pressure,
and approximating the number of preimages belonging to $\Lambda$ with continuous functions.

We will show that our method gives Cantor sets in fibers, which are obtained as subsets of intersections of Cantor
sets of large thickness (\cite{K}, \cite{H}, \cite{PT}). This will guarantee the existence of uncountably many points
having more than one preimage. We show that, still, there are uncountably many points with only one prehistory, giving thus an example
 where the number of prehistories is infinite for some points, and equal to 1 for others.

We proved in \cite{MU} that the
 stable dimension, i.e the Hausdorff dimension of the stable fibers (in our case), is greatly influenced by the
  preimage counting function. We will see how this relates to our skew product case. We will prove also that the
  local unstable manifolds (in fact even the unstable directions) depend on prehistories and will estimate the angle
  between them. We can give also the unstable dimension by using a Bowen type equation from \cite{M}. Our method gives a
  large class of examples of hyperbolic endomorphisms which behave \textbf{differently} than the hyperbolic diffeomorphisms, and also differently from constant-to-1 maps.

\textbf{Outline of Main Results:}

The main object of study is the skew product $f$ defined in (\ref{skew}), and also its generalization from
(\ref{skew-general}), together with their respective locally maximal invariant set $\Lambda$.

The main results of the paper are contained in Theorems
\ref{Cantor-fibers}, \ref{unstable}, \ref{general}, \ref{s-dim}
(and the remarks thereafter), and in Corollaries \ref{infinite-dir} and \ref{preim}.

First we will remind the notion of natural extension, hyperbolicity for non-invertible maps and the notions of
stable dimension and unstable dimension.
Then in Theorem \ref{Cantor-fibers} we will show, by using a result of \cite{K} (see also \cite{H}) about intersecting
two Cantor sets of large thicknesses, that there exists a Cantor
set $F_x$ of points having several preimages/prehistories, in each fiber $\Lambda_x$ of $\Lambda$; we will also estimate
 the thickness of $F_x$.
Moreover, this will give an example of a dynamical system $f$
where some points have two preimages in $\Lambda$, and other
points with only one preimage in $\Lambda$.

In Proposition \ref{hyperbolic} we will show that our skew product example is hyperbolic as an endomorphism,
on its basic set $\Lambda$.

Then in Theorem \ref{unstable} we will show that the unstable directions corresponding to different prehistories in $\Lambda$
do not coincide, and will estimate the angle between them.

In Theorem \ref{u-dim}, we give estimates of the unstable
dimension by using a Bowen type equation on the natural extension,
from \cite{M}. And in Theorem \ref{s-dim} we show that the stable
dimension at any point of $\Lambda$ is strictly smaller than 1 by
using a result about the inverse pressure from \cite{M-Cam}; and
we obtain estimates for the stable dimension by using
approximating continuous bounds for the preimage counting
function, based on results of \cite{MU}. Thus  we obtain
information on the set of points with more than one preimage,
vis-à-vis the set of points having only one preimage.

The hyperbolicity on $\Lambda$, the existence of Cantor sets in fibers of points with more than one prehistory,
and the disjointness of unstable directions coresponding to different prehistories are shown similarly, also
for the more general examples defined in (\ref{skew-general}), for a suifficiently small $\alpha$.
This property of strong non-invertibility is preserved for the examples in (\ref{skew-general}),
for all $\alpha>0$ small enough, by a
type of Newhouse phenomenon, involving intersecting Cantor sets of large thickness in fibers.

By combining Theorem \ref{s-dim} and Theorem \ref{unstable} we
prove in Corollary \ref{infinite-dir} that there exist points in
$\Lambda$ with infinitely many prehistories and infinitely many
different unstable manifolds. We deduce in Corollary \ref{preim}
that for the examples in (\ref{skew-general}) having contraction
factors equal to $\frac 12$ (and the rest of their parameters
being in an open set), there are uncountably many points with one
preimage in $\Lambda$ and uncountably many points with two
preimages in $\Lambda$. We prove then in Corollaries
\ref{infinite-dir} and \ref{preim} that, also from the point of
view of Hausdorff dimension, the behaviour of this last class of
examples (with contraction factors $\frac 12$ and the other
parameters in an open set), is far from that of a homeomorphism on
$\Lambda$, as well as far from that of a constant-to-1 map on
$\Lambda$.

Also it is important to remark that our invariant set $\Lambda(\alpha)$ is \textbf{not an attractor}, but instead
is a \textbf{saddle basic set}.

\section{Hyperbolic skew product endomorphisms. Unstable directions on the basic set.}

Hyperbolic smooth endomorphisms appear naturally in many instances
when invariant sets have self-intersections. Several aspects of
their dynamics are very different than in the case of
diffeomorphisms (for example \cite{MU}, \cite{Pr}, etc.) Consider
in the sequel a smooth (say $\mathcal{C}^2$) map $f:M \to M$ on a
smooth Riemannian manifold $M$ and let $\Lambda$ be a compact
invariant set, i.e $f(\Lambda) = \Lambda$. So each point of
$\Lambda$ has at least one $f$-preimage in $\Lambda$; however it
may have several $f$-preimages in $\Lambda$.

 We define now a \textbf{prehistory} of a
point $x \in \Lambda$, as an infinite sequence $\hat x = (x,
x_{-1}, x_{-2}, \ldots)$ of consecutive preimages, i.e $f(x_{-1})
= x, f(x_{-2}) = x_{-1}, \ldots$, with $x_{-i} \in \Lambda, i \ge
1$. We take then the space of all these prehistories $\hat
\Lambda$, and consider the shift homeomorphism $\hat f : \hat
\Lambda \to \hat \Lambda, \hat f(\hat x) = (f(x), x, x_{-1},
x_{-2}, \ldots), \hat x \in \hat \Lambda$. The compact space $\hat
\Lambda$ can be made a metric space in a natural way, and will be
called the \textbf{natural extension} of the couple $(\Lambda,
f)$. One can introduce also the tangent bundle over $\hat
\Lambda$, given by $T_{\hat x}:= \{(\hat x, v), v \in  T_{x}M\},
\hat x \in \hat \Lambda $. Now following \cite{R}, we
define the notion of hyperbolicity for the endomorphism $f$ as a
continuous splitting of the tangent bundle over $\hat \Lambda$
into stable directions and unstable directions depending on
prehistories. So we have a splitting $T_{\hat x}M = E^s_x \oplus
E^u_{\hat x}, \hat x \in \hat \Lambda$ where $Df_x(E^s_x)\subset
E^s_{f(x)}$ and $Df(E^u_{\hat x})\subset E^u_{\hat f \hat x}$ and
$Df$ contracts uniformly the vectors from $E^s_x$ and expands the
vectors from $E^u_{\hat x}$. We denote also by $Df_s(x):=
Df|_{E^s_x}$ and by $Df_u(\hat x):= Df|_{E^u_{\hat x}}, \hat x \in
\hat \Lambda$. Associated to this splitting, one constructs local
stable and local unstable manifolds of size $r>0$ for some small
$r$, denoted by $W^s_r(x)$, $W^u_r(\hat x)$, for
$\hat x \in \hat \Lambda$.

For a fixed $r>0$ small enough, we define the \textbf{stable
dimension} at $x\in \Lambda$, denoted by $\delta^s(x)$, as the
Hausdorff dimension $HD(W^s_r(x) \cap \Lambda)$; respectively the
\textbf{unstable dimension} at $\hat x \in \hat \Lambda$, denoted
by $\delta^u(\hat x)$, as $HD(W^u_r(\hat x) \cap \Lambda)$. We
proved in \cite{M} that the unstable dimension is equal to the
unique zero of the pressure function $t \to P_{\hat f|_{\hat
\Lambda}}(-t \log |Df_u(\hat x)|)$, where the pressure is
considered on the natural extension $\hat \Lambda$. However in the
case of the stable dimension there is no simple general formula,
due to the complicated foldings that may take place in $\Lambda$.

Let us also define the \textbf{preimage counting function}
$d(\cdot)$ associated to $f$ on the invariant set $\Lambda$,
namely $d(x):= \text{Card}\{f^{-1}x \cap \Lambda\}$. It can be
checked that $d(\cdot)$ is an upper semi-continuous function
(\cite{MU}); it is not necessarily constant, nor continuous. This
is bringing additional difficulties in estimating the stable
dimension in the non-invertible case, as noted in \cite{MU} or
\cite{MU2}.

For a small positive $\alpha$, let us take now the subintervals
$I_1^\alpha, I_2^\alpha \subset I:= [0, 1]$, of small positive length, with $I_1^\alpha = [b_1(\alpha),
b_2(\alpha)], I_2^\alpha = [b_3(\alpha), b_4(\alpha)]$; assume
that $b_2(\alpha) < \frac 12$, $b_2(\alpha)$ is very close to
$\frac 12$, and that $b_3(\alpha)$ is very close to $1-\alpha$ and
$b_3(\alpha) < 1-\alpha$; we assume that $|b_1(\alpha) - \frac 12|$
and $|b_3(\alpha) - (1-\alpha)|$ are both much smaller than
$\alpha$, say
\begin{equation}\label{small}
 0 < \max\{|b_1(\alpha)-\frac 12|,   |b_3(\alpha) - (1-\alpha)|\} := \epsilon(\alpha)< \alpha^2
\end{equation}

The intervals $I_1^\alpha, I_2^\alpha$ depend on $\alpha$. But may
also be denoted in the sequel simply by $I_1, I_2$ when dependence
on $\alpha$ is unambiguous.

Let us take also $g: I_1^\alpha \cup I_2^\alpha \to I$, a strictly
increasing smooth map which expands both $I_1^\alpha$ and
$I_2^\alpha$ to $I$, i.e $g(I_1^\alpha) = g(I_2^\alpha) = I$.
Assume that $g'(x) > \beta(\alpha)>>1, x \in I_1^\alpha \cup
I_2^\alpha$. From this dilation condition, we see that there exist
subintervals $I_{11}^\alpha, I_{12}^\alpha \subset I_1^\alpha$ and
$I_{21}^\alpha, I_{22}^\alpha \subset I_2^\alpha$ such that
$g(I_{11}^\alpha) = g(I_{21}^\alpha) = I_1^\alpha$ and
$g(I_{12}^\alpha) = g(I_{22}^\alpha) = I_2^\alpha$. Let us denote
by $J^\alpha : = I_{11}^\alpha \cup I_{12}^\alpha \cup
I_{21}^\alpha \cup I_{22}^\alpha$ and $$J_*^\alpha: = \{ x \in
J^\alpha, \ g^i x \in J^\alpha, i \ge 0 \}$$ When the dependence
on $\alpha$ is clear, we may denote these sets also by $J, J_*,
I_{i, j}, i, j = 1, 2$.

We will define now for a small $\alpha >0$, the skew product with overlaps in fibers
$f_\alpha: J_*^\alpha \times I \to J_*^\alpha \times I$
\begin{equation}\label{skew}
f_\alpha(x, y) = (g(x), h_\alpha(x, y)), \text{where}
\end{equation}
 $$h_\alpha(x, y)= \left\{ \begin{array}{ll}
                     x+\frac y2, \ x \in I_{11}^\alpha \\
                     1-x+\frac y2, \ x \in I_{21}^\alpha \\
                     1-\frac y2, \ x \in I_{12}^\alpha \\
                      \frac y2, \ x \in I_{22}^\alpha
                     \end{array}
           \right.
$$

We shall denote also the function $h_\alpha(x, \cdot): I \to I$ by
$h_{x, \alpha}$ for $x \in J_*^\alpha$. From the definition of
$h_\alpha(x, y)$ it can be seen that for $x \in J_*^\alpha \cap
I_1$, there are two images of intervals intersecting inside $\{x\}
\times I$, namely $h_{x_{-1}, \alpha}(I)$ and $h_{\tilde x_{-1},
\alpha}$, where $x_{-1}$ denotes the $g$-preimage of $x$ belonging
to $I_1^\alpha$, and $\tilde x_{-1}$ denotes the $g$-preimage of
$x$ belonging to $I_2^\alpha$.

We denote by
\begin{equation}\label{Lambda}
\Lambda(\alpha):= \mathop{\cup}\limits_{x \in J_*^\alpha}\mathop{\cap}\limits_{n \ge 0} \mathop{\cup}\limits_{y \in g^{-n}x
\cap J_*^\alpha} h_{y, \alpha}^n(I),
\end{equation}

where $h_{y, \alpha}^n:= h_{f^{n-1}y, \alpha} \circ \ldots \circ
h_{y, \alpha}, n \ge 0$. For $x \in J_*^\alpha$ we also denote by
$$\Lambda_x(\alpha):=\mathop{\cap}\limits_{n \ge 0}
\mathop{\cup}\limits_{y \in g^{-n}x \cap J_*^\alpha} h_{y,
\alpha}^n(I), $$ and call it \textbf{the fiber} of
$\Lambda(\alpha)$ over $x$. It is clear that $\Lambda(\alpha)$ is
a compact $f$-invariant set, but $\Lambda(\alpha)$ \textbf{is not
necessarily totally invariant} hence the number of
$f_\alpha$-preimages of a point from $\Lambda(\alpha)$, belonging
to $\Lambda(\alpha)$, may vary. The sets $\Lambda(\alpha),
\Lambda_x(\alpha)$ will be denoted simply by $\Lambda, \Lambda_x$
when dependence on $\alpha$ is clear.

We shall now prove that the fibers $\Lambda_x(\alpha), x \in
J_*^\alpha \cap I_1$ have an interesting property, namely they
contain a Cantor set of points which have two different
$f_\alpha$-preimages in $\Lambda(\alpha)$.

Before we proceed with the Theorem, let us remind the notion of
\textbf{thickness} of a Cantor set introduced by Newhouse
(\cite{N-carte}), and studied also for example in \cite{H},
\cite{K}, \cite{PT}.

Consider a Cantor set $K$ obtained as $I_0 \setminus
\mathop{\cup}\limits_{n\ge 1} U_n$, where $U_n$ are open
subintervals of $I_0$, called the \textbf{gaps} of $K$, and $I_0$
is the minimal interval containing $K$. Of course the gaps of $K$
can be ordered in many ways, and we call such an ordering
$\mathcal{U}=(U_n)_n$ a \textbf{presentation} of $K$. For a point
$u \in \partial U_n$, let $C$ be the connected component of $I_0
\setminus (U_1 \cup \ldots \cup U_n)$ which contains $u$; the
component $C$ is also called a \textbf{bridge} at $u$. For such a
point denote by $\tau(K, \mathcal{U}, u):=
\frac{\ell(C)}{\ell(U_n)}$ (see \cite{PT}), where $\ell(C)$
denotes the length of the subinterval $C$. Then the thickness of
$K$ is defined by $$ \tau(K) = \mathop{\sup}\limits_{\mathcal{U}}
\mathop{\inf}\limits_{u \in K}\tau(K, \mathcal{U}, u), $$ where
the infimum is taken over all the boundary points of finite gaps
of $K$, and the supremum is taken over all different presentations
of $K$. In fact it can be proved that the supremum in the
definition of thickness is attained for a presentation with
decreasing lengths of gaps, i.e so that $\ell(U_p) \le \ell(U_n)$
if $p \ge n$. Thickness is an important numerical invariant of a
Cantor set, and it is preserved by an affine transformation of the
interval. Newhouse showed that if $K_1, K_2$ are Cantor sets with
$\tau(K_1) \tau(K_2) >1$ and neither of them is contained in a gap
of the other (i.e they are \textbf{interleaved}), then $K_1 \cap
K_2 \ne \emptyset$ (see \cite{N-carte}).

We will use below the thickness in order to prove that the fibers
$\Lambda_x$ (defined in (\ref{Lambda})) contain "big"
intersections of certain Cantor sets.

\begin{thm}\label{Cantor-fibers}
In the above setting, for all points $x \in J_*^\alpha \cap I_1$, there exists a Cantor set of points $F_x(\alpha) \subset \Lambda_x(\alpha)$,
 so that each point from $F_x(\alpha)$ has two different $f$-preimages in $\Lambda$.
Also if $x \in J_*^\alpha \cap I_2^\alpha$, it follows that there is a Cantor set
$F_x(\alpha) \subset \Lambda_x(\alpha)$ such that each point of $F_x(\alpha)$ has more than one prehistory in $\hat \Lambda(\alpha)$.
\end{thm}

\begin{proof}

Fix a small positive $\alpha$; we will work with the corresponding
$I_j, f, \Lambda, h$ for this fixed $\alpha$, without recording
their dependence on it. From the construction of the subintervals
$I_{ij}, i, j = 1, 2$, we see that $I_{11}$ and $I_{12}$ are close
to $\frac 12$, while $I_{21}, I_{22}$ are close to $1-\alpha$
(recall that $\alpha>0$ was taken very small).

For a point $x \in J_*$, denote by $x_{-1, 1}, x_{-1, 2}$ the two $g$-preimages
of $x$ in $J_*$, i.e $x_{-1, 1} \in I_1, x_{-1, 2} \in I_2, g(x_{-1,
1}) = g(x_{-1, 2}) = x$.

At the first iteration, if $x \in I_1$, we obtain $\Lambda_x(1):=
h_{x_{-1, 1}}(I) \cup h_{x_{-1, 2}}(I) = [\alpha, c_1]$, with $c_1 = \frac 12 + x_{-1, 1}\approx 1$; indeed $x_{-1, 1} \in [a, c]$ and $|a - \frac 12| < \alpha$. If $x \in I_2$, then $\Lambda_x:=
h_{x_{-1, 1}}(I) \cup  h_{x_{-1, 2}}(I) = [0, 1]$, where again
$x_{-1, 1}, x_{-1, 2}$ are the two $g$-preimages of $x$, one in
$I_1$ and the other in $I_2$.

For the second iteration, we obtain $\Lambda_x(2):= h_{x_{-1,
1}}(\Lambda_{x_{-1, 1}}(1))  \cup  h_{x_{-1, 2}}(\Lambda_{x_{-1,
2}}(1)) = [\alpha, c_2]$, where $c_2 = x_{-1, 1} + c_1/2 \approx
\frac{1+ c_1}{2}$. For $x \in I_2$, we have $\Lambda_x(2):=
h_{x_{-1, 1}}(\Lambda _{x_{-1, 1}}(1)) \cup h_{x_{-1,
2}}(\Lambda_{x_{-1, 2}}(1)) = [0, \frac 12] \cup [1-\frac{c_1}{2},
1 - \frac{\alpha}{2}]$. This is hinting already on the nature of
the gaps of the fiber $\Lambda_x$, which is obtained as an
intersection $\Lambda_x = \mathop{\cap}\limits_{n \ge 1}
\Lambda_x(n)$. There will be gaps coming from the successive
iterations of the gap $(\frac 12, 1-\frac{c_1}{2})$.

At the third iteration, we see that for $x \in J_* \cap I_1$, $\Lambda_x(3):= h_{x_{-1, 1}}(\Lambda_{x_{-1, 1}}) \cup h_{x_{-1, 2}}(\Lambda_{x_{-1, 2}}) = [\alpha, c_3]$, where $c_3 = x_{-1, 1} + \frac{c_2}{2}$ and $x_{-1, i}$ is the $g$-preimage of $x$ belonging to $I_i, i = 1, 2$.
And for $x \in J_* \cap I_2$, we have $\Lambda_x(3) = h_{x_{-1, 1}}(\Lambda_{x_{-1, 1}}(2)) \cup h_{x_{-1, 2}}
(\Lambda_{x_{-1, 2}}(2)) = [0, \frac 14] \cup [\frac 12 - \frac{c_1}{4}, \frac 12 - \frac{\alpha}{4}]
\cup [1- \frac{c_2}{2}, 1 - \frac{\alpha}{2}]$.
At this iteration, we have therefore the gaps $(\frac 14, \frac 12 - \frac{c_1}{4})$ and $(\frac 12 - \frac{\alpha}{4}, 1 - \frac{c_2}{2})$, having lengths $\frac 14 - \frac{c_1}{4}$ and $\frac 12 - \frac{c_2}{2} + \frac{\alpha}{4}$ respectively. These lengths are very small in comparison to the lengths of their associated left and right bridges $C$. We can say that $\ell(U_j(3)) \le \Delta(\alpha)^{-1} \ell(J)$, where $J$ is one of the component intervals of $\Lambda_x(3)$ and $U_j(3)$ is one of the gaps between two consecutive subintervals $J$ of $\Lambda_x(3)$, and where $\Delta(\alpha) = O(\frac 1 \alpha)$. So $\Delta(\alpha) \to \infty$ when $\alpha \searrow 0$.

At the fourth iteration we see gaps forming inside $\Lambda_x(4)$
for $x \in J_* \cap I_1$ as well, i.e $\Lambda_x(4)$ contains
several disjoint closed subintervals. For $x \in J_* \cap I_1$, we
have $\Lambda_x(4) = h_{x_{-1, 1}} (\Lambda_{x_{-1, 1}}(3)) \cup
h_{x_{-1, 2}}(\Lambda_{x_{-1, 2}}(3))$; so from the above
calculations we obtain $\Lambda_x(4) = [x_{-1,1} + \alpha/2,
x_{-1, 1} + c_3/2] \cup [1-x_{-1, 2}, 1-x_{-1, 2} + \frac 18] \cup
[1-x_{-1, 2} + \frac 14 -\frac{c_1}{8}, 1-x_{-1, 2} + \frac 14 -
\frac{\alpha}{8}] \cup [1-x_{-1, 2} + \frac 12 - \frac{c_2}{4},
1-x_{-1, 2} + \frac 12 - \frac{\alpha}{4}]$. We see thus that
there are gaps forming this time in $\Lambda_x(4)$, and that still
$\ell(J)/\ell(U) \ge \Delta(\alpha)$, for any component $J$ of
$\Lambda_x(4)$ and adjacent gap $U$ between two consecutive $J$'s.
If $x \in J_* \cap I_2$, then we obtain $\Lambda_x(4)$ as before
and we see more gaps forming, still having the property that
$\ell(J)/\ell(U) \ge \Delta(\alpha)$ for any component subinterval (bridge)
$J$ and adjacent gap $U$.

Assume now that at iteration $n$, $\Lambda_x(n) = J_1(x, n) \cup \ldots J_{k(n)}(x, n)$, where this is an ordered union
 of mutually disjoint closed subintervals, i.e the right endpoint of $J_k(x, n)$ is strictly less than the left endpoint
 of $J_{k+1}(x, n), k = 1, \ldots, k(n)-1$. Denote $J_k(x, n) = [a_k(x, n), c_k(x, n)], k = 1, \ldots, k(n)$ if $x \in I_1$, and $J_k(x, n) = [\tilde a_k(x, n), \tilde c_k(x, n)], k = 1, \ldots, k(n)$ for $x \in I_2$. Then we will show that there must exist points from $\Lambda_x$ as close as we
 want to the endpoints of each $J_k(x, n)$.
Indeed we know from above that there are gaps of $\Lambda_y$ for $y \in J_* \cap I_2$, as close as we want to 0.
But if $x \in I_1$ and $x_{-1, 2}$ is its $g$-preimage in $I_2$, we have that $h_{x_{-1, 2}}$ takes $\Lambda_{x_{-1, 2}}$ into $\Lambda_{x}$, so there exist points of $\Lambda_{x}$ as close as we want to the left endpoint of $J_1(x, n), n >1$.
Again for $x \in I_1$, and the $g$-preimage $x_{-1, 1} \in I_1$ of $x$, points $z$ in $\Lambda_{x_{-1, 1}}$ are taken by $h_{x_{-1, 1}}$ into points of type $x_{-1, 1}+\frac {z}{2}$; if we repeat the procedure, we see that there are points from $\Lambda_x$ as close as we want to the right endpoint of $J_{k(n)}(x, n)$. Now from the fact that there are points of $\Lambda_{x_{-1,1}}$ as close as we want to the left endpoint of intervals of type $J_1(x_{-1, 1}, n)$, we see that by applying $g_{x_{-1, 1}}$ we obtain points from $\Lambda_{x}$ as close as we want to the right endpoint of $J_{k(n)}(x, n)$ if $x \in I_2$.

But now by recalling that all the intervals of type $J_k(x, n)$ are obtained by applying repeatedly $h_y$ (for preimages $y$ of $x$), we obtain that there are points from $\Lambda_x$ as close as we want to each of the endpoints of the subintervals $J_k(x, n)$ obtained at step $n$. This procedure  tells us that indeed, we can use the subintervals $J_k(x, n)$ as bridges in the construction of the respective Cantor set, since the gaps between them do not extend inside any of $J_k(x, n)$. Thus we can use the lengths of the subintervals of type $J_k(x, n), k = 1, \ldots k(n)$ and the lengths of the gaps between them, in the calculation of the thickness of $\Lambda_x$.

We also notice the following property: assume that $\Lambda_x(n) =
J_1(x, n) \cup \ldots \cup J_{k(n)}(x, n)$ for $x \in J_*$, where
these disjoint subintervals are arranged in increasing order. Then
if $x \in J_* \cap I_1$, we claim that the left endpoint of
$J_1(x, n)$, i.e $a_1(x, n)$ is $\epsilon$-close to $\alpha$ where
for our fixed $\alpha$, $\epsilon$ denotes the positive number
$\epsilon(\alpha)$ defined in (\ref{small}). Also if $x \in J_*
\cap I_2$, we claim that the left endpoint of $J_1(x, n)$, i.e
$\tilde a_1(x, n)$ is 0 and the right endpoint of $J_{k(n)}(x,
n)$, i.e $\tilde c_{k(n)}(x, n)$ is $\epsilon(\alpha)$-close to $1
- \frac{\alpha}{2}$, for every $n \ge 1$. We saw this property for
iterations 1 through 4, let us prove it in general by induction.
Assume it is satisfied at step $n$ for $\Lambda_x(n), \forall x
\in J_*$. Then at step $n+1$, if $x \in I_1$ we have
$\Lambda_x(n+1) = h_{x_{-1, 1}}(\Lambda_{x_{-1, 1}}(n)) \cup
h_{x_{-1, 2}}(\Lambda_{x_{-1, 2}}(n))$;
 so $J_1(x, n+1) = 1- x_{-1, 2} + \frac 12 \cdot J_1(x_{-1, 2}, n)$. Since by induction the left endpoint
 of $J_1(x_{-1, 2}, n)$ is 0, we see that the left endpoint of $J_1(x, n+1)$ is $1-x_{-1, 2}$.
 But $x_{-1, 2} \in I_{21}$ is $\epsilon$-close to $1-\alpha$, hence the left endpoint of
 $J_1(x, n+1)$ is $\epsilon$-close to $\alpha$. Now if $x \in I_2$, we see that the left endpoint
 $J_1(x, n+1)$ is obtained from applying $y \to y/2$ to the interval $J_1(x_{-1, 2}, n)$, thus it is equal to 0.
 And the right endpoint of $J_{k(n)}(x, n+1)$ is obtained from applying the map $y \to 1-y/2$ to $J_1(x_{-1, 1}, n)$;
 thus the right endpoint of $J_{k(n)}(x, n+1)$ is $\epsilon$-close to $1-\frac{\alpha}{2}$.

We have to see also what is happening to the subintervals that
overlap in $\Lambda_x(n+1)$ for $x \in I_1$. The overlap is between
the points $x_{-1, 1} + \frac{a_1(x_{-1, 1}, n)}{2}$ and $1- x_{-1, 2}+\frac{\tilde c_{k(n)}(x_{-1, 2}, n)}{2}$.
However we saw above that $a_1(x_{-1, 1}, n)$ is $\epsilon$-close to $\alpha$ (recall that $\epsilon = \epsilon(\alpha)$);
 and that $\tilde c_{k(n)}(x_{-1, 2}, n)$ is $\epsilon$-close to $1 - \frac{\alpha}{2}$. We recall also that $I_1, I_2$ are $\epsilon$-close to $\frac 12$ and $1-\alpha$ respectively. Thus the overlap mentioned above is between points that are $\epsilon$-close to $\frac 12 + \frac \alpha 2$ and $\alpha + \frac 12 - \frac \alpha 4 = \frac 12 + \frac{3\alpha}{4}$. As for all $\alpha>0$ small we have $\frac 12 + \frac{\alpha}{2} < \frac 12 + \frac{3\alpha}{4}$, we see that indeed we have overlaps inside $\Lambda_x(n), x \in J_* \cap I_1, n >1$.
Denote by $J_{int}$ this maximal overlap at the third
Iteration. Then $J_{int}$ comes from applying $h_{x_{-1, 1}}$ to
the intervals in the last half of  $\Lambda_{x_{-1, 1}}(n-1)$ and
from applying $h_{x_{-1, 2}}$ to the intervals in the first half
of $\Lambda_{x_{-1, 2}}(n-1)$. However we noticed that, when $y \in I_1 \cap J_*$, the
component intervals of $\Lambda_y(n-1)$ outside $[1/2 + \alpha/2,
1/2 + 3\alpha/4]$ contain each a Cantor set, obtained successively by
eliminating a fixed proportion of the intervals at step $n-1$. Thus there exists $n$ large enough so that, for any two points
$\xi, \zeta \in \Lambda_x \cap J_{int}$, there exists a gap $U_1$ in $\Lambda_{x_{-1, 1}}(n-1)$
and a gap $U_2$ in $\Lambda_{x_{-1, 2}}(n-1)$ with $h_{x_{-1,
1}}(U_1) \cap h_{x_{-1, 2}}(U_2)$ non-empty and situated between
$\zeta$ and $\xi$. Obviously $n$ depends on the distance between
$\xi$ and $\zeta$.

We saw above that there are points from $\Lambda_x$ as close as we
wish to the endpoints of the subintervals $J_k(x, n), k = 1, \ldots, k(n)$ of $\Lambda_x(n), n \ge 1$. Therefore these subintervals can be used as bridges $J$ and the intervals between them as gaps $U$, in the Cantor set construction of $\Lambda_x, x \in J_*$. By induction we also see that at each step $n$ we have

$$\ell(J_k(x, n))/\ell(U) \ge \Delta(\alpha), k = 1, \ldots,
k(n)$$

for any subinterval $J_k(x, n)$ of $\Lambda_x(n)$ and any corresponding adjacent gap $U$ of $J_k(x, n)$ (where we say that $U$ is an \textbf{adjacent gap} for $J_k(x, n)$ if it
is immediately at the left or at the right of $J_k(x, n)$ at step
$n$).
But this ratio between a subinterval and its adjacent gap is
preserved by linear transformations, like the ones we deal with in
$h_x, x \in J_*$. By overlapping two subintervals (as it may
happen in $\Lambda_x(n), x \in J_* \cap I_1$), we can only
increase the lengths of bridges $J$ and decrease the lengths of
adjacent gaps $U$. Thus we see by induction that at each step
\begin{equation}\label{delta}
\ell(J)/\ell(U) \ge \Delta(\alpha),
\end{equation}

for each bridge subinterval $J$ and adjacent gap $U$ of $\Lambda_x(n)$.
Since by applying iterations we cut in half the length of the gaps between the subintervals $J_k(x, n), k = 1, \ldots, k(n)$ at step $n$ (or decrease them by an even larger factor), we obtain that the gaps are ordered decreasingly when $n \nearrow \infty$; this presentation of $\Lambda_x$ will be denoted by
$\mathcal{U}_{step}$. And from the observation made when we
defined the thickness of a Cantor set, we have that
$\tau(\Lambda_x) = \mathop{\inf}\limits_{\xi \in \Lambda_x}
\tau(\Lambda_x, \mathcal{U}_{step}, \xi)$. But from (\ref{delta})
we notice that

\begin{equation}\label{tau}
\tau(\Lambda_x) \ge \Delta(\alpha), x \in J_*
\end{equation}

Now let $x \in J_* \cap I_1$; then $\Lambda_x = h_{x_{-1,
1}}(\Lambda_{x_{-1, 1}}) \cup h_{x_{-1, 2}}(\Lambda_{x_{-1, 2}})$,
where $x_{-1, 1} \in g^{-1}(x) \cap I_1$ and $x_{-1, 2} \in
g^{-1}(x) \cap I_2$. It is easy to see that the two Cantor sets
$S_1(x):= h_{x_{-1, 1}}(\Lambda_{x_{-1, 1}})$ and $S_2(x) :=
h_{x_{-1, 2}}(\Lambda_{x_{-1, 2}})$ are interleaved, i.e neither
set is contained in a gap of the other. Also from (\ref{tau}) we
know that $\tau(S_1(x)) \ge \Delta(\alpha)$ and $\tau(S_2(x)) \ge
\Delta(\alpha)$.

We recall that $\Delta(\alpha) \mathop{\to}\limits_{\alpha \to 0} \infty$.
 Thus if $\tau_1:= \tau(S_1(x)), \tau_2 :=
\tau(S_2(x))$ and $\alpha$ is small enough, we can check that $\tau_1 \tau_2 >1$ and in addition
that $(\tau_1, \tau_2) \in B$, where: $$ B:= \{(\tau_1, \tau_2),
\tau_1 > \frac{\tau_2^2+3\tau_2+1}{\tau_2^2} \ \text{or} \ \tau_2
> \frac{\tau_1^2+3\tau_1+1}{\tau_1^2}\} \cap \{(\tau_1, \tau_2),
\tau_1> \frac{(1+2\tau_2)^2}{\tau_2^3} \ \text{or} \ \tau_2 >
\frac{(1+2\tau_1)^2}{\tau_1^3}\} $$

Since $S_1(x), S_2(x)$ are not interleaved and we do have the above conditions satisfied for $\Delta(\alpha)$ large enough (i.e for $\alpha$ small enough), we obtain from \cite{K} that there exists indeed a Cantor set $F_x$ in $S_1(x) \cap S_2(x)$.
Moreover from \cite{H} (page 882), it follows that we have:
$$\tau(F_x) \ge \sqrt{\text{min}\{\tau(S_1(x)),
\tau(S_2(x))\}} \ge \sqrt{\Delta(\alpha)})$$

From the definition of the sets $S_1(x), S_2(x)$ we see that for every $x \in I_1 \cap J_*$, each point from the set $F_x$ has two different $f$-preimages belonging to $\Lambda$.

Now if $x \in J_* \cap I_2$, there exists a $g$-preimage $x_{-1, 1} \in I_1$ of $x$ and then in the fiber $\Lambda_{x_{-1, 1}}$ there must exist a Cantor set $F_{x_{-1, 1}}$ of points having two distinct $f$-preimages in $\Lambda$. So we obtain that there exists a Cantor set $F_x:= h_{x_{-1, 1}}(\Lambda_{x_{-1, 1}})$ of points with two different $f^2$-preimages in $\Lambda$.
The Theorem is then proved.

\end{proof}

In fact we will prove in Corollary \ref{infinite-dir} that there are points with infinitely many prehistories in $\hat \Lambda$.
Let us prove now the hyperbolicity of $f$ as an endomorphism on the basic set $\Lambda$. We will show that this set is of saddle type, i.e $f$ has both stable and unstable directions on it.

Although it may appear at a first sight that $f$ is expanding horizontally, the calculation on derivatives shows this to be false. Indeed we have the derivative of $f$,
$$
Df(x, y) = \left(\begin{array}{ll}
                    g'(x) & 0 \\
                    \partial_x h(x, y) & \partial_y h(x,y)
                  \end{array} \right),
$$ where $\partial_x h(x, y)$ represents the partial derivative of
$h$ with respect to $x$ at the point $(x, y)$. Thus for a vector $
\bar w = (0, v) \in \mathbb R \times \mathbb R$, we get $Df(x,
y)\cdot \bar w = \left(\begin{array}{cc}
                                 0 \\
                                  \partial_y h(x, y) v
                                \end{array} \right) $,
so the vector space $\{(0, v) \in \mathbb R \times \mathbb R\}$ is invariant and $Df$ is contracting on vertical lines; these vertical lines represent therefore the stable tangent subspaces.

However if we consider the horizontal vector $\bar w = (\zeta, 0) \in \mathbb R \times \mathbb R$, then
$Df(x, y) \cdot \bar w = \left(\begin{array}{cc}
                                     g'(x) \cdot \zeta \\
                                     \partial_x h(x, y) \cdot\zeta
                                    \end{array} \right)$,
so the horizontal line $\{(\zeta, 0), \zeta \in \mathbb R, 0 \in
\mathbb R\}$ is not invariated by $Df$, and thus the unstable
spaces do not have to be equal to this line.

To prove the hyperbolicity of the non-invertible map $f$ on $\Lambda$ and to construct its unstable spaces, we use a generalization
of a theorem of Newhouse (\cite{N-art}) to this endomorphism case.

Let $f:M \to M$ be a smooth, not
necessarily invertible map and suppose that $\Lambda$ is a
compact $f$-invariant set in $M$. Assume that there exists a field
of cones in the tangent space,
 $\mathcal{C} = \{C_{\hat z}\}_{\hat z \in \hat \Lambda}$, so that the dimension of the core
 linear space of $C_{\hat z}$ is constant on $\hat \Lambda$; but the cone field $\mathcal{C}$ is not
 necessarily assumed to be $Df$-invariant.
Let us say that a function $f$ is \textbf{expanding and
co-expanding on the cone field} $\mathcal{C}$, if given the
notations:

$$ m_{\mathcal{C}, \hat z}(f):= \mathop{\inf}\limits_{v \in
C_{\hat z}, v \ne 0}\frac{|Df_z v|}{|v|}, \text{and} \ \ m_{\mathcal{C}, \hat z}'(f):= \mathop{\inf}\limits_{v \notin
C_{\hat f \hat z}} \frac{|Df^{-1}_{fz}v|}{|v|}, \hat z \in \hat
\Lambda, $$

we have that $\mathop{\inf}\limits_{\hat z \in
\hat \Lambda} m_{\mathcal{C}, \hat z}(f) > 1$, and
$\mathop{\inf}\limits_{\hat z \in \hat \Lambda} m_{\mathcal{C},
\hat z}'(f) > 1$.

\begin{unthm}[Newhouse]\label{cone}
In the above setting, assume that there exists an integer $N \ge 1$ such that $f^N$ is expanding and co-expanding on $\mathcal{C}$; then it follows that $f$ is hyperbolic on $\Lambda$.
\end{unthm}

The proof is similar to the one given in \cite{N-art}.
We can prove consequently the following result of hyperbolicity for our skew product:

\begin{pro}\label{hyperbolic}

In the above setting, i.e with $g:J_* \to J_*$ expanding and $h:I_*
\times I \to J_* \times I$ contracting in the second coordinate
over the invariant set $\Lambda$ defined above, we
have that $f(x, y) = (g(x), h(x, y))$ is uniformly hyperbolic as an endomorphism on $\Lambda$.
\end{pro}

\begin{proof}

Let a continuous positive function $\gamma$ defined on $\hat
\Lambda$, and the cone $C^u_{\hat z} := \{(v, w) \in
\mathbb R^2, |w| \le \gamma(\hat z) \cdot |v|\}, z = (x, y) \in \Lambda, \hat z \in \hat \Lambda$. The
dimension of the core real linear space of this cone is 1.
Our cone field will be then $\mathcal{C}^u = \{C^u_{\hat z}\}_{\hat z \in \hat \Lambda}$ and we do not know a priori whether it is $Df$-invariant.
We have $Df_z(v,w) = \left(\begin{array}{c}
                            g'(x)\cdot v \\
                            \partial_x h(z) \cdot v + \partial_y h(z) \cdot w
                           \end{array}\right)$.
So in order to have an $f$-expanding field of cones, it is enough
to take $$ |g'(x)|^2 > 1+\gamma^2(\hat z), z \in \Lambda $$

If we assume $|g'(x)| > \beta > 1, x \in X$, then it would be enough to have

\begin{equation}\label{gamma}
0 < \gamma(\hat z) \le \sqrt{\beta^2 - 1}, \ \text{or}, \ 0 < \gamma(\hat z) \le \sqrt{\beta^{2N} - 1},
\end{equation}
where the second inequality is needed if we work with $f^N$
instead of $f$. So in this last case, $f^N$ is expanding on the
cone field $\mathcal{C}^u$.

Now we estimate the co-expansion coefficient.
If $N\ge 1$ is an integer and if $(v, w) \notin C^u_{\hat f^N \hat z}$, then $|w| > \gamma(\hat f^N \hat z)\cdot |v|$.
 Denote also $f^N z = (g^N(x), h(f^{N-1}x, h_{N-1}(x, y)))$, where
$h_{N-1}$ is given by: $f^{N-1} z = (g^{N-1} (x), h_{N-1}(x, y)), z = (x, y) \in \Lambda$. So
\begin{equation}\label{der}
\aligned
&\partial_x h_N (x, y) =  \partial_x h (f^{N-1}x,
h_{N-1}(x, y)) \cdot \partial_x g^{N-1}(x) +
\partial_y h (f^{N-1} x, h_{N-1}(x, y)) \cdot \partial_x h_{N-1}(x, y) = \\
& =\partial_x h (f^{N-1}x, h_{N-1}(x, y)) \cdot \partial_x
g^{N-1}(x) +
\partial_y h (f^{N-1} x, h_{N-1}(x, y)) \cdot \partial_x h (f^{N-2}x, h_{N-2}(x, y)) \cdot \partial_x g^{N-2}(x)
+\\
&\partial_y h (f^{N-1} x, h_{N-1}(x, y)) \cdot \partial_y
h(f^{N-2} x, h_{N-2}(x, y)) \cdot \partial_x h_{N-2}(x, y)
\endaligned
\end{equation}

Denote by $K:= \mathop{\sup}\limits_\Lambda |\partial_x h|$ and
$K':= K \cdot \frac{1}{1-\delta/\beta}$, where $\delta \in (0, 1)$
is a contraction factor, such that $|\partial_y h | < \delta < 1$ on $\Lambda$.
Therefore by induction in (\ref{der}) we have:

\begin{equation}\label{in}
|\partial_x h_N(x, y)| \le K \cdot |(g^{N-1})'x| + \delta K \cdot
|(g^{N-2})'x| + \ldots \le K' \cdot |(g^{N-1})' x|
\end{equation}

But $D(f^N)^{-1}_{f^N z} \left( \begin{array}{c}
                                  v \\

                                  w
                                \end{array} \right)  = \left( \begin{array}{c}
               \frac{v}{(g^N)' (x)}  \\
              \frac{-\partial_x h_N(z) v}{(g^N)'(x) \cdot \partial_y h_N(z)} + \frac{w}{\partial_y h_N(z)}
           \end{array} \right)$.

Hence $||D(f^N)^{-1}_{f^N z} \left(\begin{array}{c}
                                    v \\

                                    w
                                  \end{array}\right) ||^2
\ge \frac{v^2}{|(g^N)'(x)|^2} \left(1 + \frac{|\partial_x
h_N|^2(z)}{|\partial_y h_N|^2(z)} \right) + \frac{w^2}{|\partial_y
h_N|^2(z)} \cdot \left( 1 - \frac{2|\partial_x h_N(z)|}{|(g^N)'(x)
\cdot \gamma(\hat f^N \hat z)|} \right )$, for any $N \ge 1$. But
then since $|\partial_x h_N(z)| \le K' \cdot |(g^{N-1})'(x)|$, and
$K'$ depends only on $g, h$, there must exist $N$ sufficiently
large such that $\left| \frac{2K'\cdot
\frac{1}{g'(g^{N-1}x)}}{\gamma(\hat f^N(\hat z))} \right| \le
\frac{2K'}{\beta \cdot\sqrt{\beta^{2N}-1}} < \frac 1 2$, if we
take the map $\gamma(\cdot)$ to be constant on $\hat \Lambda$ and
close to $\sqrt{\beta^{2N}-1}$ (although smaller than
$\sqrt{\beta^{2N}-1}$). So:
\begin{equation}\label{inverse}
||D(f^N)^{-1}_{f^N z} \left(\begin{array}{c}
                                    v \\

                                    w
                                  \end{array}\right) ||^2
\ge \frac{|w|^2}{2 |\partial_y h_N(z)|^2}
\end{equation}

But we had $|w| > \gamma(\hat f^N \hat z) \cdot |v|$, hence:

$$ \frac{|w|^2}{2|\partial_y h_N(z)|^2} \ge \frac{|w|^2}{2
\delta^{2N}} > \frac{|v|^2 + |w|^2}{\delta}, $$

for $N$ sufficiently large, since $\gamma(\cdot)$ is bounded on
$\hat \Lambda$. Hence $\mathop{\inf}\limits_{\hat z \in \hat
\Lambda} m_{\mathcal{C}^u, \hat z}'(f^N) > 1$ for some large
integer $N$. So $f^N$ is both expanding and co-expanding on the
cone field $\mathcal{C}^u$ over $\Lambda$, so the map $f$ is
hyperbolic according to the previous result. The unstable space
corresponding to the arbitrary prehistory $\hat z \in \hat
\Lambda$ is obtained then as $E^u_{\hat z} =
\mathop{\cap}\limits_{n \ge 0} Df^n(C^u_{\hat f^{-n} \hat z})$.

\end{proof}

We will show now that there correspond different unstable
tangent directions $E^u_{\hat z} \ne E^u_{\hat z'}$, to two different prehistories of $z$, $\hat z, \hat z' \in \hat \Lambda$.

\begin{thm}\label{unstable}
Let the skew product $f: J_* \times I \to J_* \times I$ defined
above in (\ref{skew}). Then if $\hat z$ and $\hat z'$ are two
different prehistories of $z$ from $\hat \Lambda$, it follows that
$E^u_{\hat z} \ne E^u_{\hat z'}$; in particular the local unstable
manifolds $W^u_r(\hat z), W^u_r(\hat z')$ are not tangent to each other.
\end{thm}

\begin{proof}

Let the point $z = (x, y) \in \Lambda_x, x \in J_*$. For a tangent
vector $(v, w)$ we have $$Df_z \left( \begin{array}{c}
                                  v \\

                                  w
                                \end{array} \right)   =
\left(\begin{array}{c}
         g'(x)\cdot v \\
     \partial_x h(z) \cdot v + \partial_y h(z) \cdot w
       \end{array}\right)$$

Assume that the unstable tangent space corresponding to the prehistory $\hat z \in \hat \Lambda$ is given by
\begin{equation}\label{u-space}
E^u_{\hat z} := \{(v, \omega(\hat z) \cdot v), v \in \mathbb R\}, \hat z \in \hat \Lambda,
\end{equation}
where $\omega(\cdot)$ is a bounded function on $\hat \Lambda$, since the unstable spaces must be
transversal to the stable (vertical) ones.
Thus from the above formula, we obtain $$Df_z  \left( \begin{array}{c}
                                  v \\

                                  w
                                \end{array} \right)  =
\left(\begin{array}{c}
         g'(x)\cdot v \\
     \partial_x h(z) \cdot v + \partial_y h(z) \cdot \omega(\hat z) v
       \end{array}\right)$$

Now we know from the construction of unstable spaces on $\hat
\Lambda$ as intersections of iterates of unstable cones (Proposition \ref{hyperbolic}) that $Df_z(E^u_{\hat z}) \subset
E^u_{\hat f \hat z}$. Therefore $\partial_x h(z) +  \partial_y
h(z) \cdot \omega(\hat z)  = \omega(\hat f \hat z) \cdot g'(x),
\hat z \in \hat \Lambda$. So if $z_{-1}$ denotes an $f$-preimage
of $z$ belonging to $\Lambda$ and $x_{-1}$ denotes an $g$-preimage
of $x$ belonging to $J_*$, we obtain the following recurrence
formula for $\omega(\cdot)$:

\begin{equation}\label{omega}
\omega(\hat f \hat z) = \frac{1}{g'(x)} \partial_x h(z) +
\frac{1}{g'(x)}\partial_y h(z) \cdot \omega(\hat z)
\end{equation}

By iterating (\ref{omega}) and by recalling that $\omega$ is a
bounded function on $\hat \Lambda$ (since the stable and
unstable directions must be transversal to each other),  we
obtain:
\begin{equation}
\aligned
\omega(\hat f\hat z) & = \frac{1}{g'(x)}\partial_x h(z) +
\frac{1}{g'(x)}\cdot
\partial_y h(z) (\frac{1}{g'(x_{-1})}\cdot
\partial_xh(z_{-1}) + \frac{1}{g'(x_{-1})} \cdot
\partial_yh(z_{-1}) \cdot \omega(\hat z_{-1})) \\
& = \frac{1}{g'(x)} \partial_x h(z) +
\frac{1}{g'(x)g'(x_{-1})}\partial_y h(z) \partial_x h(z_{-1}) +
\frac{1}{g'(x)g'(x_{-1})} \partial_y h (z) \partial_y h(z_{-1})
\omega(\hat z_{-1})\\ & = \ldots = \frac{1}{g'(x)}
\partial_x h(z) + \mathop{\sum}\limits_{i=1}^\infty \frac{1}{g'(x)
\cdot \ldots \cdot g'(x_{-i})} \cdot \partial_x h(z_{-i})
\partial_y h(z_{-i+1}) \ldots \partial_y h(z)
\endaligned
\end{equation}

Now if $\hat z'= (z, z_{-1}', z_{-2}', \ldots)$ is another
prehistory of $z$ from $\hat \Lambda$, say with $z_{-1}' \ne
z_{-1}$, we have from above that $$ \omega(\hat z') =
\frac{1}{g'(x_{-1}')}\partial_x h(z_{-1}') +
\mathop{\sum}\limits_{i=2}^\infty \frac{1}{g'(x_{-1}') \ldots
g'(x_{-i}')} \cdot \partial_x h(z_{-i}')
\partial_y h(z_{-i+1}') \ldots \partial_y h(z_{-1}')$$

Therefore
\begin{equation}\label{l}
\aligned
&\omega(\hat z) - \omega(\hat z') = \frac{1}{g'(x_{-1})}
\partial_x h(z_{-1}) - \frac{1}{g'(x_{-1}')}
\partial_xh(z_{-1}') + \\
& + \left[ \mathop{\sum}\limits_{i=2}^\infty \frac{\partial_x
h(z_{-i})\cdot
\partial_y h(z_{-i+1}) \ldots \partial_yh(z_{-1})}{g'(x_{-1})
\ldots g'(x_{-i})}  - \mathop{\sum}\limits_{i=2}^\infty
\frac{\partial_xh(z_{-i}')\cdot
\partial_y h(z_{-i+1}') \ldots \partial_yh(z_{-1}')}{g'(x_{-1}') \ldots g'(x_{-i}')} \right]
\endaligned
\end{equation}

Let us assume that $$\beta^2 > g'(x) > \beta >>1, x \in J_*,$$ for
some large $\beta$ which depends on $\alpha$; this holds since the
map $g$ expands the small intervals $I_1, I_2$ to the whole $I =
[0, 1]$, and $g$ was assumed increasing.

We assume that $|\partial_yh| < \delta < 1$ on $\Lambda$.
Now the expression in the straight brackets in (\ref{l}), is less
than $\frac{1}{\beta^2}(1 + \frac{\delta}{\beta} + \ldots) =
\frac{1}{\beta^2}\cdot \frac{1}{1-\frac{\delta}{\beta}} <
\frac{1.2}{\beta^2}$, if $\beta$ is large enough.

But if $z_{-1}, z_{-1}'$ are different preimages of $z$, it
follows that we must have $z_{-1} \in \Lambda_{x_{-1}}$ and
$z_{-1}' \in \Lambda_{x_{-1}'}$ for two different $g$-preimages of
$x$, say $x_{-1} \in I_1$ and $x_{-1}' \in I_2$. Then we have
$\partial_xh(z_{-1}) = 1$ and $\partial_x h(z_{-1}')= -1$.
Therefore we have that $$|\frac{1}{g'(x_{-1})}
\partial_x h(z_{-1}) - \frac{1}{g'(x_{-1}')}
\partial_xh(z_{-1}')| > \frac{2}{\beta^2}$$

Hence from the above estimate of the expression in the straight
brackets of (\ref{l}), we see that for two prehistories $\hat z,
\hat z' \in \hat \Lambda$ of $z$ with  $z_{-1} \ne z_{-1}'$, we
obtain: $$ |\omega(\hat z) - \omega(\hat z')| >
\frac{0.7}{\beta^2}$$

In general, let $\hat z, \hat z'$ two different prehistories of
$z$ from $\hat \Lambda$. Then there exists $m \ge 1$ so that
$z_{-m} \ne z_{-m}'$ and $z_{-i} = z_{-i}'$ for $i = 0, \ldots,
m-1$ (where as always we denote $z = z_0$). Then similarly as
above we obtain that
\begin{equation}\label{omega-gen}
 |\omega(\hat z) - \omega(\hat z')| >
\frac{0.7}{\beta^{m+1}},
\end{equation}
where we recall that $\beta = \beta(\alpha)$. Clearly if
$E^u_{\hat z} \ne E^u_{\hat z'}$ for two different prehistories of
$z$ then also the corresponding local unstable manifolds
$W^u_r(\hat z), W^u_r(\hat z')$ are different, and they are not
tangent to each other.

\end{proof}

We estimated thus in the previous Theorem the angle between
unstable directions corresponding to different prehistories of the
same point, by using the dilation factor $\beta$ of $g: I_1 \cup
I_2 \to I$.

We will give now a generalization of example (\ref{skew}) to a
family which is nonlinear in $x$ and linear in $y$, with fiber
contraction factors possibly different from $\frac 12$.

First let us fix a small positive $\alpha$. Then take the
subintervals $I_1^\alpha, I_2^\alpha \subset I = [0 , 1]$ so that
$I_1^\alpha$ is contained in $[\frac 12 - \epsilon(\alpha), \frac
12 + \epsilon(\alpha)]$ and $I_2^\alpha$ is contained in
$[1-\alpha - \epsilon(\alpha), 1-\alpha + \epsilon(\alpha)]$, for
some small $\epsilon(\alpha) < \alpha^2$. We take then a strictly
increasing smooth map $g:I_1^\alpha\cup I_2^\alpha \to I$ with
$\beta^2
> g'(x)
> \beta >> 1, x \in I_1^\alpha \cup I_2^\alpha$,  and define the subintervals $I_{ij}^\alpha \subset I_i^\alpha$ such that
$g(I_{ij}^\alpha) = I_j^\alpha, i, j = 1,2$. As in (\ref{skew}),
we define the set $J_*^\alpha$ as the set of points from
$I_1^\alpha \cup I_2^\alpha$ whose $g$-iterates remain in the same
set.  Now consider the skew product $f_\alpha: J_*^\alpha \times I
\to J_*^\alpha \times I$,

\begin{equation}\label{skew-general}
f_\alpha(x, y) = (g(x), h_\alpha(x, y)), \text{with}
\end{equation}
$$h_\alpha(x, y)= \left\{ \begin{array}{ll}
                     \psi_{1, \alpha}(x) + s_{1, \alpha} y, \ x \in I_{11}^\alpha \\
                     \psi_{2, \alpha}(x) + s_{2, \alpha} y, \ x \in I_{21}^\alpha \\
                     \psi_{3, \alpha}(x) - s_{3, \alpha} y, \ x \in I_{12}^\alpha \\
                      s_{4, \alpha} y, \ x \in I_{22}^\alpha,
                     \end{array}
           \right.
$$

where for some small $\vp_0$, we take $s_{1, \alpha}, s_{2,
\alpha}, s_{3, \alpha}, s_{4, \alpha}$ to be positive numbers,
$\vp_0$-close to $\frac{1}{2}, \frac 12, \frac 12, \frac 12$
respectively; and $\psi_{1, \alpha}(\cdot), \psi_{2,
\alpha}(\cdot), \psi_{3, \alpha}(\cdot)$ are smooth (say
$\mathcal{C}^2$) functions on $I$ which are $\vp_0$-close in the
$\mathcal{C}^1$-metric, to the linear functions $x \to x$, $x \to
1-x$ and $x \to 1$, respectively.

By $|g_1 - g_2|_{\mathcal{C}^1}$ we shall denote the distance in
the $\mathcal{C}^1(I)$-metric between two smooth functions on I,
$g_1$ and $g_2$.

We shall denote also the function $h_\alpha(x, \cdot): I \to I$ by
$h_{x, \alpha}(\cdot)$, for $x \in J_*^\alpha$. Again when
$\alpha$ is unambiguous and fixed, we will not record the
dependence on $\alpha$ (but will keep it in mind).

\begin{thm}\label{general}
There exists a function $\vartheta(\alpha)>0$ defined for all
positive small enough numbers $\alpha$, with $\vartheta(\alpha)
\mathop{\to}\limits_{\alpha \to 0} 0$ and such that, if $f$ is an
arbitrary map defined in (\ref{skew-general}) whose parameters
satisfy:
\begin{equation}\label{vart}
\max \left\{|\psi_1(x)-x|_{\mathcal{C}^1}, |\psi_2(x) -1
+x|_{\mathcal{C}^1}, |\psi_3(x)- 1|_{\mathcal{C}^1}, |s_1-\frac
12|, |s_2-\frac 12|, |s_3 - \frac 12|, |s_4 - \frac 12|\right\} <
\vartheta(\alpha)
\end{equation}
 then we obtain:

a) For $x \in J_* \cap I_1$, there exists a Cantor set $F_x
\subset \Lambda_x$, s. t every point of $F_x$ has two different $f$-preimages in
$\Lambda$. And if $x \in J_* \cap I_2$, then there exists a Cantor set $F_x \subset \Lambda_x$ s. t every point
of $F_x$ has two different $f^2$-preimages in $\Lambda$.

b) $f$ is hyperbolic on $\Lambda$.

c) If $\hat z, \hat z' \in \hat \Lambda$ are two different
prehistories of an arbitrary point $z \in \Lambda$, then
$E^u_{\hat z} \ne E^u_{\hat z'}$. We have the same estimate for
the angle between $E^u_{\hat z}$ and $E^u_{\hat z'}$ as in
(\ref{omega-gen}).

\end{thm}

\begin{proof}

The proof uses basically the same ideas as in Theorem \ref{Cantor-fibers},
Proposition \ref{hyperbolic} and Theorem \ref{unstable}, with certain modifications.

a) \ $I_1, I_2$ are
$\epsilon(\alpha)$-close to $\frac 12, 1-\alpha$. We assumed in
the definition (\ref{skew-general}), that
$\epsilon(\alpha) < \alpha^2 << 1$. Like in the proof of Theorem
\ref{Cantor-fibers}, the set $\Lambda_x(n)$ is formed by $k(n)$
disjoint subintervals $J_1(x, n), \ldots, J_{k(n)}(x, n)$ arranged
in an increasing order; assume also that for $x \in J_* \cap I_1$,
$J_k(x, n) = [a_k(x, n), c_k(x, n)], k = 1, \ldots, k(n)$ and for
$x\in J_* \cap I_2$, $J_k(x, n) = [\tilde a_k(x, n), \tilde c_k(x,
n)], k = 1, \ldots, k(n)$.
Also for a point $x \in J_*$, there exist two $g$-preimages of $x$ in $J_*$, which will be denoted by $x_{-1, 1}$ (for the $g$-preimage of $x$ from $I_1$), and $x_{-1, 2}$ (for the $g$-preimage of $x$ belonging to $I_2$).

Then by induction we see that $\tilde a_1(x, n) = 0$ for all $n >
1$ and $x \in J_* \cap I_2$. This implies that for all $x \in J_*
\cap I_1, n >1$, we have $a_1(x, n) = h_{x_{-1, 2}}(\tilde
a_1(x_{-1, 2}, n-1)) = h_{x_{-1, 2}}(0)$. Therefore we obtain that $a_1(x,
n)$ is $\alpha^2$-close to $\psi_2(1-\alpha)$ if $x \in J_* \cap
I_1$. This implies also that the right most endpoint of $\Lambda_{y}(n), y \in I_2 \cap J_*$, namely $\tilde c_{k(n)}(y, n)$, is $\alpha^2$-close to $\psi_3(\frac 12)- s_3 \psi_2(1-\alpha)$.

We see now that the right endpoint of $J_{k(n)}(x, n), x \in I_1
\cap J_*$, namely $c_{k(n)}(x, n)$ is $\alpha^2$-close to
$\psi_1(\frac 12) + s_1 c_{k(n-1)}(x_{-1, 1}, n-1)$. Thus we see
that $c_{k(n)}(x, n) \nearrow c(x, \infty)$ when $n \to \infty$
and that $c(x, \infty)$ is $\alpha^2$-close to $\psi_1(\frac 12) (1 +s_1 + s_1^2+\ldots)= \frac{\psi_1(\frac
12)}{(1-s_1)}$. It follows as in Theorem \ref{Cantor-fibers} that
we do have points of $\Lambda_x$ as close as we want to
$c_{k(n)}(x, n)$ for $n>1$; hence there are points of $\Lambda_x$
as close as we want to $c(x, \infty)$ when $x \in J_* \cap I_1$.

Now if $a_1(x, n)$ is $\epsilon(\alpha)$-close to
$\psi_2(1-\alpha)$ for $x \in I_1 \cap J_*$, it follows from
construction that $\tilde c_{k(n)}(x', n)$ is
$\epsilon(\alpha)$-close to $\psi_3(\frac 12)-
s_3\psi_2(1-\alpha)$ if $x' \in I_2 \cap J_*$. Recall also that
$\epsilon(\alpha) < \alpha^2$. Again, since there are points of
$\Lambda_{x}$ as close as we want to $a_1(x, n)$ when $x \in
I_1\cap J_*$, we see that $a_1(x, n)=h_{x_{-1, 2}}(0) \in
\Lambda_x$;  thus there exist points of $\Lambda_{x'}$ as close as
we want to $\tilde c_{k(n)}(x, n) = \psi_3(\frac 12)- s_3 \cdot
a_1(x'_{-1, 1}, n)$ for $x' \in I_2\cap J_*$.

Now for an iteration of order $n >3$, we will want to have the two
phenomena which gave the fractal structure of $\Lambda_x$ in
Theorem \ref{Cantor-fibers}. The first desired phenomenon is the
overlapping in $\Lambda_x, x \in I_1\cap J_*$, of the first
intervals of $h_{x_{-1, 1}}(\Lambda_{x_{-1, 1}})$ (i.e
subintervals $J_k(x, n), k < m(n)$ for some $m(n)$), with the last
intervals of $h_{x_{-1, 2}}(\Lambda_{x_{-1, 2}})$ (i.e with the
subintervals of type $J_k(x, n), k > p(n)$). And the second
desired phenomenon is the appearance of a gap of length comparable
to $\alpha$, in $\Lambda_x, x \in I_2 \cap J_*$, between
$h_{x_{-1, 2}}(\tilde c_{k(n-1)}(x_{-1, 2}, n-1))$ and $h_{x_{-1,
1}}(c(x_{-1, 1}, \infty))$. We use $c(x_{-1, 1}, \infty)$ since
$c_{k(n)}(x_{-1, 1}, n) \mathop{\nearrow}\limits_{n \to \infty}
c(x_{-1, 1}, \infty)$ and there are points from $\Lambda_{x_{-1,
1}}$ as close as we want to $c(x_{-1, 1}, \infty)$; thus the gap
in $\Lambda_x$ is bounded above by $h_{x_{-1, 1}}(c(x_{-1, 1},
\infty))$.

Therefore our two conditions are satisfied if:
\begin{equation}\label{C1}
\psi_1(\frac 12) + s_1\cdot \psi_2(1-\alpha) < \psi_2(1-\alpha) +
s_2 \cdot [\psi_3(\frac 12)-s_3 \psi_2(1-\alpha)], \ \text{and}
\end{equation}
\begin{equation}\label{C2}
\frac{\alpha}{2}< \psi_3(\frac 12)- s_3\cdot \frac{\psi_1(\frac
12)}{1-s_1}- s_4\cdot [\psi_3(\frac 12)- s_3 \psi_2(1-\alpha)] < \alpha
\end{equation}

As we can see these two conditions are verified if there exists a
sufficiently small $\vartheta(\alpha)>0$, s. t
\begin{equation}\label{var}
|s_i - \frac 12| < \vartheta(\alpha), i = 1, \ldots, 4 \
\text{and} \ \max \left\{ |\psi_1(x)- x|_{\mathcal{C}^1},
|\psi_2(x)- 1+ x|_{\mathcal{C}^1}, |\psi_3(x)-
1|_{\mathcal{C}^1}\right\} < \vartheta(\alpha)
\end{equation}

It is clear that $\vartheta(\alpha) \mathop{\to}\limits_{\alpha
\to 0} 0$. If $\vartheta(\alpha)$ is small enough, then the
thickness of the fibers $\Lambda_x$ remains larger than
$\Delta(\alpha)$. This permits for every $x \in I_1 \cap J_*$, to
have an intersection between the images $h_{x_{-1, 1}}(\Lambda_{x_{-1, 1}})$ and
$h_{x_{-1, 2}}(\Lambda_{x_{-1, 2}})$ inside $\Lambda_x$. Thus
we obtain a Cantor set $F_x \subset \Lambda_x$ s.t every point of
$F_x$ has two different $f$-preimages inside $\Lambda$. We obtain again that $$\tau(F_x) \ge \sqrt{\Delta(\alpha)}, x \in I_1 \cap J_*,$$
where $\Delta(\alpha) = O(\frac{1}{\alpha}), \alpha >0$.

b) \ The hyperbolicity follows in the same way as in Proposition
\ref{hyperbolic}, if the inequalities in (\ref{var}) are satisfied
and $\vartheta(\alpha)$ is small enough.

c) The disjointness of unstable directions corresponding to
different prehistories of the same point follows as in the proof
of Theorem \ref{unstable}, since the derivatives of $\psi_1,
\psi_2$ are $\vartheta(\alpha)$-close to $1$, respectively $-1$.
And similarly we obtain the same estimates (\ref{omega-gen}) for
the angle between two unstable directions $E^u_{\hat z}$ and
$E^u_{\hat z'}$, corresponding to two different prehistories of
the same point $z \in \Lambda$;  we recall also that $\beta$
depends on $\alpha$, in (\ref{omega-gen}).

\end{proof}

Let us study now the unstable, respectively stable dimension on
$\Lambda$; from this study we shall obtain \textbf{further
information} about the preimage counting function on $\Lambda$. We
will see more arguments towards the idea that the skew products
defined in (\ref{skew-general}) are \textbf{far} from being
homeomorphisms on their respective basic set $\Lambda$, and also
\textbf{far} from being 2-to-1 on $\Lambda$.

First we have the following result about the unstable dimension:

\begin{thm}\label{u-dim}
For a small fixed $\alpha>0$, let the function $f:\Lambda \to
\Lambda$ defined in (\ref{skew-general}). Then the unstable
dimension $\delta^u(\hat z) = t^u, \forall \hat z \in \hat
\Lambda$, where $t^u$ is the unique zero of the pressure function
$t \to P_{\hat f|_{\hat \Lambda}}(t\Phi^u)$, and where
$\Phi^u(\hat y) := -\log|Df_u(\hat y)|, \hat y \in \hat \Lambda$.
Consequently if $g'(x) > \beta(\alpha) >>1$ on $J_*$, we have
$$\delta^u(\hat z) < \frac{\log 2}{\log\frac{\beta(\alpha)}{2}},
\hat z \in \hat \Lambda$$
\end{thm}

\begin{proof}
The first part of the Theorem follows from \cite{M} since the unstable manifolds in our case are 1-dimensional
(hence conformal).
So $\delta^u(\hat z) = t^u$, for all $\hat z \in \hat \Lambda$. We notice that $t^u$ is the zero of a pressure function
calculated on the
natural extension $\hat \Lambda$.

Now, from the proof of Theorem \ref{general} we know that
$\omega(\hat z) < \frac{1}{\beta(\alpha)}, \hat z \in \hat
\Lambda$. Hence $|Df_u(\hat z)| > \frac{\beta(\alpha)}{2}, \hat z
\in \hat \Lambda$ so $\Phi^u(\hat z) < -\log
\frac{\beta(\alpha)}{2}, \hat z \in \hat \Lambda$.

Also it is easy to see that $h_{top}(f|_\Lambda) = \log 2$ since
the Bowen balls of $f$ are given mainly by the expansion of $g$ in
the horizontal direction, and $g|_{J_*}$ is conjugated to
$\sigma_2$ on the one-sided Bernoulli shift $\Sigma_2^+$.
Therefore $0 = P_{\hat f|_{\hat \Lambda}}(t^u\Phi^u) < -t^u\log
\frac{\beta(\alpha)}{2} + \log 2$ and we obtain
$
\delta^u(\hat z) = t^u < \frac{\log
2}{\log\frac{\beta(\alpha)}{2}} $.
\end{proof}

Now we want to estimate the stable dimension over $\Lambda$; by contrast to Theorem \ref{u-dim} or to the diffeomophism case, we do not know here that $\delta^s(z)$ is constant when $z$ ranges in $\Lambda$.

Recall that we denoted by $d(\cdot)$ the \textit{preimage counting
function} for $f$ on $\Lambda$, defined by $d(z):= \text{Card}
\{z' \in \Lambda, f(z') = z\}, z \in \Lambda$. One major
difficulty is that $d(\cdot)$ \textbf{is not necessarily
continuous} on $\Lambda$, and not necessarily constant. So the
expression $P(t \Phi^s - \log d)$ does not make sense, since
pressure was defined for continuous maps. We will overcome this
obstacle in 2 different ways. The first one will be by using a
notion of \textbf{inverse pressure} (see \cite{MU2}). For a continuous non-invertible map $f:X \to X$ on a compact metric space $X$, the inverse
pressure $P_f^-(\cdot)$ is a functional defined by using consecutive
preimages of points (rather than forward iterates like for
usual pressure). It is useful in the case of estimating the stable
dimension of endomorphisms that are not necessarily constant-to-1
(\cite{M-Cam}, \cite{MU2}). For the negative function $\Phi^s$, we
proved that there exists a unique zero $t_s^-$ of the function $t
\to P_f^-(t \Phi^s)$, and that in our case $$\delta^s(z) \le t^-_s, z
\in \Lambda$$

The second way is by using \textbf{continuous upper bounds}
$\eta(\cdot)$ for the preimage counting function $d(\cdot)$, and
then to employ the unique zero $t_\eta$ of the function $t \to
P(t\Phi^s - \log \eta)$.

\begin{thm}\label{s-dim}
Let a sufficiently small $\alpha>0$ and a function $f$ defined as in
(\ref{skew-general}) for the chosen $\alpha$,
 and assume that the parameters of $f$ satisfy condition (\ref{vart}).

a) Then the stable dimension $\delta^s(z) \le t^-_s < 1$, for any point $z \in \Lambda$.

b) If $\eta(\cdot)$ is a continuous function on $\Lambda$ such
that $d(z) \le \eta(z), z \in \Lambda$, it follows that
$\delta^s(z) \ge t_\eta, z \in \Lambda$, where $t_\eta$ is the
unique zero of the function $t \to P(t\Phi^s - \log \eta)$.
\end{thm}

\begin{proof}

a) We take a fixed small enough $\alpha>0$; this will imply that $\vartheta(\alpha)$ is also small enough such that Theorem \ref{general} works.

We can apply the estimate for the stable dimension from
\cite{M-Cam}, since from Theorem \ref{Cantor-fibers}, an arbitrary
fiber $\Lambda_x$ does not contain intervals. So no local stable
manifold is contained in $\Lambda$, thus $\Lambda$ is not a local
repellor. We proved in Theorem 1 from \cite{M-Cam} that
$\delta^s(z) \le t_s^-$, where $t^-_s$ is the unique zero of the
inverse pressure function $t \to P^-(t\Phi^s)$. We showed also
that $t_s^- < 1$ (adapted to our case), as $f$ is conformal on
vertical fibers, hyperbolic on $\Lambda$ and does not have
critical points in $\Lambda$.

b) It follows immediately from \cite{MU} since our map $f$ is
conformal on stable manifolds and does not have critical
points in $\Lambda$.

\end{proof}

\textbf{Remarks:} 1) This Theorem shows that, when the fiber
contraction factors $s_1, s_2, s_3, s_4$ are all equal to $\frac
12$, $f|_\Lambda$ is far from being a homeomorphism (from the
point of view of stable dimension), for \textbf{all} choices of parameters
$\psi_j, j = 1, 2, 3$, that satisfy condition (\ref{vart}). Indeed
if $\delta^s(z)$ were the zero $t_s$ of the pressure $t \to
P(t\Phi^s)$, then since $\Phi^s \equiv -\log 2$ on $\Lambda$, we
would obtain $t^s = 1$. But we saw in Theorem \ref{s-dim} that
$\delta^s(z) < 1, z \in \Lambda$. Hence when the contraction on
fibers is identically $\frac 12$ and the other parameters of $f$
belong to some specified open set, we do not have a formula for
the stable dimension like in \cite{F}.

2) Also we notice above that $f|_\Lambda$ is not 2-to-1 either.
Indeed assume that the parameters $s_i, \psi_j, i = 1, \ldots, 4,
j = 1, 2$ satisfy (\ref{vart}) and $\alpha$ is small enough and
let the function $f$ given by these parameters in
(\ref{skew-general}) (we do not record now the dependence of $f$
on $\alpha$, but are keeping it in mind). Then from Theorem
\ref{general}, we know that $f$ is hyperbolic as an endomorphism
on $\Lambda$.
 in that case, from \cite{MU2} it follows that for all $z
\in \Lambda$, $\delta^s(z)$ would be equal to the unique zero
$t^s_2$ of $t \to P(t\Phi^s - \log 2)$. But since
$h_{top}(f|_\Lambda) = \log 2$, it would follow that $t^s_2 = 0$.
On the other hand we know from the proof of Theorem
\ref{Cantor-fibers} that there are points $z \in \Lambda_x$ with
$\delta^s(z) = HD(\Lambda_z) \ge \frac{\log 2}{\log (2 +
\frac{1}{\sqrt{\Delta(\alpha)}})} > 0$. We obtained a
contradiction, hence $f|_\Lambda$ is not 2-to-1, $\hfill \square$

\

In fact if $s_i = \frac 12, i = 1, \ldots, 4$ in the definition
(\ref{skew-general}) and if there were only at most $m$
prehistories in $\hat \Lambda$ for each point of $\Lambda$ then
the stable dimension would still be equal to the zero $t_s$ of the
pressure $t \to P(t\Phi^s)$; this follows from \cite{MU2}.
Therefore there must exist a point $z_m\in \Lambda$ with more than
$m$ prehistories, for any integer $m >1$. But then we obtain a
sequence of points $(z_m)_m$ in $\Lambda$ each one with more than
$m$ prehistories in $\hat \Lambda$. They must accumulate to some
point $z \in \Lambda$ which will have therefore infinitely many
prehistories in $\Lambda$. By employing also Theorem
\ref{unstable}, we have proved thus:

\begin{cor}\label{infinite-dir}
Let a small $\alpha >0$ and a function $f$ as in
(\ref{skew-general}), and assume that $s_i = \frac 12, i = 1,
\ldots, 4$ and that the parameters $\psi_j, j = 1, \ldots, 3$
satisfy (\ref{vart}). Then there exists a point $z \in \Lambda$
having infinitely many different prehistories in $\hat \Lambda$,
 and thus infinitely many different unstable directions $E^u_{\hat z}$.
\end{cor}

We prove now that our examples are both far from having a
homeomorphism-type behaviour, and also far from the constant-to-1
maps of \cite{MU2}:

\begin{cor}\label{preim}
Let a small $\alpha>0$ and a function $f$ defined as in
(\ref{skew-general}), s. t the parameters $s_i, \psi_j, i = 1, \ldots, 4, j = 1,\ldots, 3$ of $f$ satisfy (\ref{vart}). Write $\Lambda$ as the union $V_1 \cup V_2$, where
$V_1$ is defined as the set of points having only one $f$-preimage inside
$\Lambda$ and $V_2$ is the set of points having exactly two
$f$-preimages in $\Lambda$.

a) \ Then $\delta^s(z) \in (\frac{\log
2}{\log(2+\frac{1}{\Delta(\alpha)})}, 1), z \in \Lambda$. So if
$\alpha$ tends to 0, then the stable dimension at an arbitrary
point of $\Lambda$ may be made as close as we want to 1, but
always strictly smaller than 1.

b) \ $V_1$ is an open uncountable set in $\Lambda$, and $V_2$ is a
closed set in $\Lambda$.

c) \ Assume moreover that in the definition (\ref{skew-general}) of $f$,
the contraction factors $s_i, i = 1, \ldots, 4$ are all equal to
$\frac 12$. Then $V_2$ is uncountable as well.
\end{cor}

\begin{proof}

a) \ We apply a result of Palis and Takens (\cite{PT}) giving an
estimate for the Hausdorff dimension of  a Cantor set $K$ in the
line, as follows: $$HD(K) \ge \frac{\log
2}{\log(2+\frac{1}{\tau(K)})}$$

We showed in Theorem \ref{general} that we have $\tau(\Lambda_x)
\ge \Delta(\alpha)$; and as in the proof of Theorem
\ref{Cantor-fibers}, $\Delta(\alpha) \mathop{\to}\limits_{\alpha
\to 0} \infty$. Hence combining also with Theorem \ref{s-dim}, we
obtain the estimates.

b) \  We recall that $d(z)$ was defined as the number of
$f$-preimages of $z$ belonging to $\Lambda$.

We can partition now $\Lambda$ into two subsets, $V_1:= \{z \in
\Lambda, d(z) = 1\}$ and $V_2:= \{z \in \Lambda, d(z) = 2\}$. It
can be seen easily that $V_1$ is open and $V_2$ is closed, since
$d(\cdot)$ is upper semi-continuous. We remark that $V_1, V_2$ are
not necessarily $f$-invariant.

If $V_1$ would be countable, then we can approximate any point from $\Lambda$ by points from $V_2$;
given an arbitrary $n>1$, we can even approximate any point $z \in \Lambda$ with points $w$ such that
$w \in V_2$ and all its preimages of order less than $n$ are also in $V_2$; i. e if $w_{-i} \in \Lambda \cap f^{-i}w$,
then $w_{-i} \in V_2, i < n$. This makes the proof of Theorem 3.1 of \cite{MU2} to work (holomorphicity in that Theorem
is not essential); all that is important is that $f$ be conformal on stable manifolds, and this is satisfied in our case
as the stable manifolds are 1-dimensional.

Hence it would follow that $\delta^s(z) = t^s_2$, where $t^s_2$ is the unique zero of the pressure
$t \to P(t\Phi^s - \log 2)$.
Now $h_{top}(f|_\Lambda) = \log 2$, since the expansion takes place mainly in the horizontal direction and
since $g|_{J_*}$ is topologically conjugate to $\sigma_2$ on the one-sided Bernoulli space $\Sigma_2^+$.
Thus it follows that $t^s_2 = 0$.

However we saw in part a) that $HD(\Lambda_x) \ge \frac{\log
2}{\log(2+\frac{1}{\Delta(\alpha)})} > 0$. So we cannot have
$HD(\Lambda_x) = t^s_2$, and we obtain a contradiction. Therefore
$V_1$ is \textbf{uncountable}.

c) \ Now assume that all the contraction factors of $f$ on fibers
are equal to $\frac 12$. Let us suppose also that $V_2$ is
countable. Then as above, for any $n
> 1$ we can approximate any point of $\Lambda$ with points having
only one preimage of order $n$ and the proof of Theorem 3.1 from
\cite{MU2} gives that $\delta^s(z) = t^s, z \in \Lambda$, where
$t^s$ is the unique zero of the pressure function $t \to
P(t\Phi^s)$. But since in our case $\Phi^s \equiv -\log 2$ on
$\Lambda$, we obtain $t^s = 1$. But from Theorem \ref{s-dim}, $t^s
< 1$, thus a contradiction. In conclusion $V_2$ is
\textbf{uncountable} as well.
\end{proof}

\

\textbf{Email:}  Eugen.Mihailescu\@@imar.ro

Institute of Mathematics of the Romanian Academy, P. O. Box 1-764,

RO 014700, Bucharest, Romania.

Webpage: www.imar.ro/\~~mihailes

\end{document}